%% file: stability.tex
\documentclass[a4paper]{amsart}

\usepackage{amsmath,amstext,amssymb,mathrsfs,amscd,amsthm,indentfirst}
\usepackage{amsfonts}

\usepackage{booktabs}
\usepackage{verbatim}
\usepackage{enumerate}

\usepackage{graphicx}
\numberwithin{equation}{section}

\newcommand{\bN}{\mathbb{N}}

\newcommand{\bQ}{\mathbb{Q}}
\newcommand{\bR}{\mathbb{R}}

\newcommand{\bZ}{\mathbb{Z}}

\newcommand\lra{\longrightarrow}

\newcommand\Diff{\mathrm{Diff}}
\newcommand\Emb{\mathrm{Emb}}

\newcommand\colim{\operatorname*{colim}}

\newcommand\Coker{\operatorname*{Coker}}
\newcommand\Ker{\operatorname*{Ker}}

\newcommand{\gen}{\mathrm{span}}

\newcommand{\Lk}{\mathrm{Lk}}

\newcommand{\R}{\bR}

\newcommand{\Int}{\mathrm{int}}

\newcommand{\IM}{\mathrm{Im}}

\renewcommand{\epsilon}{\varepsilon}

\newcommand{\Gr}{\mathrm{Gr}}

\newcommand{\MM}{\mathscr{M}}

\mathchardef\ordinarycolon\mathcode`\:
\mathcode`\:=\string"8000
\begingroup \catcode`\:=\active
  \gdef:{\mathrel{\mathop\ordinarycolon}}
\endgroup

\usepackage[all]{xy}
\SelectTips{cm}{10}
\CompileMatrices
\usepackage{amsthm}
\theoremstyle{plain}

\newtheorem{theorem}{Theorem}[section]

\newtheorem{proposition}[theorem]{Proposition}
\newtheorem{lemma}[theorem]{Lemma}

\newtheorem{corollary}[theorem]{Corollary}

\theoremstyle{definition}
\newtheorem{definition}[theorem]{Definition}

\theoremstyle{remark}

\newtheorem{remark}[theorem]{Remark}
\newtheorem*{remark*}{Remark}

\title[Homological stability]{Homological stability for moduli spaces
  of high dimensional manifolds}

\author{S{\o}ren Galatius} 
\thanks{S. Galatius was partially supported
  by NSF grant DMS-1105058 and both authors were supported by ERC
  Advanced Grant No.\ 228082, and the Danish National Research
  Foundation through the Centre for Symmetry and Deformation.}
\email{galatius@stanford.edu}
\address{Department of Mathematics\\
  Stanford University\\
  Stanford CA, 94305}

\author{Oscar Randal-Williams}
\email{o.randal-williams@math.ku.dk}
\address{Institut for Matematiske Fag\\
Universitetsparken 5\\
DK-2100 K{\o}benhavn {\O}\\
Denmark}

\begin{document}
\begin{abstract}
  We prove a homological stability theorem for the moduli spaces of
  manifolds diffeomorphic to $\#^g S^n \times S^n$, provided $n >
  2$. This generalises Harer's stability theorem for the homology of
  mapping class groups. Combined with previous work of the authors, it
  gives a calculation of the homology of these moduli spaces in a
  range of degrees.
\end{abstract}
\maketitle

\input{chap1}
\input{chap2}
\input{chap3}
\input{chap4}
\input{chap5}

\bibliographystyle{amsalpha}
\bibliography{biblio}

\end{document}

%% file: chap1.tex
\section{Introduction and statement of results}

A famous result of Harer (\cite{H}) established \emph{homological
  stability} for mapping class groups of oriented surfaces.  For example, if
$\Gamma_{g,1}$ denotes the group of isotopy classes of diffeomorphisms
of an oriented connected surface of genus $g$ with one boundary
component, then the natural homomorphism $\Gamma_{g,1} \to
\Gamma_{g+1,1}$ induces an isomorphism in group homology
\begin{equation*}
  H_k(\Gamma_{g,1}) \lra H_k(\Gamma_{g+1,1})
\end{equation*}
as long as $g \geq (3k+2)/2$.  (Harer proved this for $g \geq 3k-1$,
but the range was later improved by Ivanov (\cite{Ivanov}) and Boldsen
(\cite{Boldsen}), see also \cite{R-WResolution}.)  This result can be
interpreted in terms of moduli spaces of Riemann surfaces, and has
lead to a wealth of research in topology and algebraic geometry.  We
prove an analogous homological stability result for moduli spaces of
manifolds of higher (even) dimension. The precise result requires the
following definition, where we assume $N \geq 2n$ and embed $S^{2n-1}
\subset \R^{2n} \subset \R^N$ in the usual way.
\begin{definition}
  Let $\MM_g(\R^N) = \MM_g^n(\R^N)$ denote the set of compact
  $2n$-dimensional submanifolds $W \subset [0,\infty) \times \R^N$
  such that $\partial W = \{0\} \times S^{2n-1}$ and $[0,\epsilon)
  \times S^{2n-1} \subset W$ for some $\epsilon > 0$, and such that
  $W$ is diffeomorphic relative to its boundary to the manifold
  $W_{g,1} = \#^g (S^n \times S^n) - \Int(D^{2n})$.  Topologise
  $\MM_g(\R^N)$ as a quotient of the space of embeddings $W_{g,1}
  \hookrightarrow [0,\infty) \times \R^N$ (with fixed behaviour near
  the boundary).

  For $N = \infty$ we write $\MM_g = \colim_{N \to \infty}
  \MM_g(\R^N)$.  Furthermore, we pick once and for all a (collared)
  embedding of the cobordism $W_{1,2} = S^n \times S^n - \Int(D^{2n}
  \amalg D^{2n})$ into $[0,1] \times \R^N$.  For $N \gg n$ all such
  embeddings are isotopic, and induce a well defined homotopy class
  of maps $\MM_g \to \MM_{g+1}$.
\end{definition}
The space $\MM_g$ is a model for the classifying space
$B\Diff^\partial (W_{g,1})$ of the topological group of
diffeomorphisms of $W_{g,1}$ fixing a neighbourhood of the boundary.
For $n = 1$, this is an Eilenberg--MacLane space $K(\Gamma_{g,1},1)$,
and hence Harer's result states that the stabilisation map $\MM_g^1
\to \MM_{g+1}^1$ induces an isomorphism in homology in a range.  Our
main result generalises this to higher $n$ (although we exclude the
case $n=2$).
\begin{theorem}\label{thm:main}
  For $n > 2$ the stabilisation map 
  \begin{equation*}
    H_k(\MM_{g}) \lra H_k(\MM_{g+1})
  \end{equation*}
  is an isomorphism for $k \leq (g-4)/2$.
\end{theorem}

By the universal coefficient theorem, stability for homology implies
stability for cohomology; in the surface case, Mumford
(\cite{Mumford}) conjectured an explicit formula for the stable
rational cohomology, which in our notation asserts that a certain ring
homomorphism
\begin{equation*}
  \bQ[\kappa_1, \kappa_2, \dots] \lra H^*(\MM_g^1;\bQ)
\end{equation*}
is an isomorphism for $g \gg *$.  Mumford's conjecture was proved in a
strengthened version by Madsen and Weiss (\cite{MW}).

Theorem~\ref{thm:main} and our previous paper \cite{GR-W2} allow us to
prove results analogous to Mumford's conjecture and the Madsen--Weiss
theorem for the moduli spaces $\MM_g^n$ with $n > 2$. The analogue of
the Madsen--Weiss theorem for $\MM_g^n$ concerns the homology of the
limiting space $\MM_\infty^n = \colim_{g \to \infty} \MM_g^n$. There
is a certain infinite loop space $\Omega^\infty MT\theta^n$ and a
continuous map
$$\alpha : \MM_\infty^n \lra \Omega^\infty MT\theta^n$$
given by a parametrised form of the Pontrjagin--Thom construction, and
in \cite[Theorem 1.1]{GR-W2} we proved that $\alpha$ induces an
isomorphism between the homology of $\MM_\infty^n$ and the homology of
the basepoint component of $\Omega^\infty MT\theta^n$.  It is easy to
calculate the rational cohomology ring of a component of
$\Omega^\infty MT\theta^n$, and hence of $\MM^n_g$ in a range of
degrees by Theorem~\ref{thm:main}.  The result is
Corollary~\ref{cor:higher-Mumford} below, which is a
higher-dimensional analogue of Mumford's conjecture.

As explained in \cite{GR-W2}, we can associate to each $c \in
H^{k+2n}(BSO(2n))$ a cohomology class $\kappa_c \in
H^{k}(\Omega^\infty MT\theta^n)$.  Pulling it back via $\alpha$ and
all the stabilisation maps $\MM_g^n \to \MM_\infty^n$ defines classes
$\kappa_c \in H^k(\MM_g^n)$ for all $g$, sometimes called
``generalised MMM classes''.  In real cohomology these classes can
equivalently be defined as follows.  Suppose $\omega \in
\Omega^{k+2n}(\Gr_{2n}^+(\R^{N+1}))$ is a differential form and
$\sigma: \Delta^k \to \MM_g^n(\R^N)$ is a smooth map given by a
$(k+2n)$-dimensional manifold $W_\sigma \subset \Delta^k \times \R^N$
fibering over $\Delta^k$.  The fibrewise tangent spaces give a map
$\tau_\sigma: W_\sigma \to \Gr_{2n}^+(\R^{N+1})$ and hence a
differential form $\tau_\sigma^* \omega \in \Omega^{k+2n}(W_\sigma)$,
and we define
\begin{equation*}
  \kappa_c(\sigma) = \int_{W_\sigma}(\tau_\sigma^* \omega) \in \R.
\end{equation*}
We have defined a linear map
\begin{equation*}
  \Omega^{k + 2n}(\Gr_{2n}^+(\R^{N+1})) \lra
  C^k_{\mathrm{sm}}(\MM_g^n(\R^N);\R)
\end{equation*}
which by Stokes' theorem is a chain map (at least for $k \geq 0$).
Hence, it induces a map of cohomology which in the limit $N \to
\infty$ sends $c \in H^{k+2n}(BSO(2n);\R)$ to $\kappa_c \in
H^k(\MM_g^n;\R)$.  The following result is our higher-dimensional
analogue of Mumford's conjecture.
\begin{corollary}\label{cor:higher-Mumford}
  Let $n > 2$ and let $\mathcal{B}\subset H^*(BSO(2n);\bQ)$ be the set
  of monomials in the classes $e, p_{n-1}, \dots, p_{\lceil
    \frac{n+1}4 \rceil}$, of degree greater than $2n$.  Then the
  induced map
  \begin{equation*}
    \bQ[\kappa_c \,\,|\,\, c \in \mathcal{B}] \lra H^*(\MM_g;\bQ)
  \end{equation*}
  is an isomorphism in the range $\ast \leq (g-4)/2$.
\end{corollary}

For example, if $n = 3$, the set $\mathcal{B}$ consists of monomials
in $e$, $p_1$ and $p_2$, and therefore $H^*(\MM_g^3;\bQ)$ agrees for
$\ast \leq (g-4)/2$ with a polynomial ring in variables of degrees 2,
2, 4, 6, 6, 6, 8, 8, 10, 10, 10, 10, 12, 12, \dots.

Our methods are similar to those used to prove many homological
stability results for homology of discrete groups, namely to use a
suitable action of the group on a simplicial complex.  For example,
Harer used the action of the mapping class group on the \emph{arc
  complex} to prove his homological stability result.  In our case the
relevant groups are not discrete, so we use a simplicial space
instead---the full diffeomorphism group of $W_{g,1}$ plays the same
role for our stability result as the mapping class group in Harer's
(similar to the situation in \cite{R-WResolution}).

Independently, Berglund and Madsen (\cite{BerglundMadsen}) have
obtained a result similar to our Theorem \ref{thm:main}, for rational
cohomology in the range $k \leq \min(n-3, (g-6)/2)$.

%%% Local Variables: 
%%% mode: latex
%%% TeX-master: "stability"
%%% End: 

%% file: chap2.tex
\section{Techniques}
\label{sec:techniques}

In this section we collect the technical results needed to establish
high connectivity of the relevant simplicial spaces.  The main results
are Theorem~\ref{thm:simplex-wise-injective} and
Corollary~\ref{cor:serre-microf-connectivity}.

\subsection{Cohen--Macaulay complexes}
\label{sec:CM}

Recall from \cite[Definition 3.4]{HW} that a simplicial complex $K$ is
\emph{weakly Cohen--Macaulay} of dimension $n$ if it is
$(n-1)$-connected and the link of any $p$-simplex is
$(n-p-2)$-connected.  In this case, we write $wCM(K) \geq n$.
\begin{lemma}
  If $wCM(X) \geq n$ and $\sigma < X$ is a $p$-simplex, then
  $wCM(\Lk(\sigma)) \geq n-p-1$.
\end{lemma}
\begin{proof}
  By assumption, $\Lk(\sigma)$ is $((n-p-1)-1)$-connected.  If $\tau <
  \Lk(\sigma)$ is a $q$-simplex, then
  \begin{equation*}
    \Lk_{\Lk(\sigma)}(\tau) = \Lk_X(\sigma \ast \tau)
  \end{equation*}
  is $((n-p-1) - q - 2)$-connected, since $\sigma \ast \tau$ is a
  $(p+q+1)$-simplex, and hence its link in $X$ is
  $(n-(p+q+1)-2)$-connected.
\end{proof}
\begin{definition}
  Let us say that a simplicial map $f: X \to Y$ of simplicial
  complexes is \emph{simplexwise injective} if its restriction to each
  simplex of $X$ is injective, i.e.\ the image of any $p$-simplex of
  $X$ is a (non-degenerate) $p$-simplex of $Y$.
\end{definition}
\begin{lemma}
  Let $f: X \to Y$ be a simplicial map of simplicial complexes.  Then
  the following conditions are equivalent.
  \begin{enumerate}[(i)]
  \item\label{item:1} $f$ is simplexwise injective,
  \item\label{item:2} $f(\Lk(\sigma)) \subset \Lk(f(\sigma))$ for
    all simplices $\sigma < X$,
  \item\label{item:3} $f(\Lk(v)) \subset \Lk(f(v))$ for all vertices
    $v \in X$,
  \item \label{item:4} The image of any 1-simplex in $X$ is a
    (non-degenerate) 1-simplex in $Y$.
  \end{enumerate}
\end{lemma}
\begin{proof}\
  \begin{itemize}
  \item[(\ref{item:1}) $\Rightarrow$~(\ref{item:2})] If $\sigma =
    \{v_0, \dots, v_p\}$ and $v \in \Lk(\sigma)$, then $\{v,v_0,
    \dots, v_p\} < X$ is a simplex, and therefore $\{f(v), f(v_0),
    \dots, f(v_p)\} < Y$ is a simplex.  Since $f$ is simplexwise
    injective, we must have $f(v) \not\in f(\sigma)$, so $f(v) \in
    \Lk(f(\sigma))$.
  \item[(\ref{item:2}) $\Rightarrow$~(\ref{item:3})] Trivial.
  \item[(\ref{item:3}) $\Rightarrow$~(\ref{item:4})] Let $\sigma =
    \{v_0, v_1\} < X$ be a $1$-simplex, and assume for contradiction
    that $f(v_0) = f(v_1)$.  Then we have $v_1 \in \Lk(v_0)$ but
    $f(v_1) = f(v_0) \not \in \Lk(f(v_0))$, contradicting $f(\Lk(v_0))
    \subset \Lk(f(v_0))$.
  \item [(\ref{item:4}) $\Rightarrow$~(\ref{item:1})] Let $\sigma =
    \{v_0, \dots, v_p\} < X$ be a $p$-simplex and assume for
    contradiction that $f|\sigma$ is not injective.  This means that
    $f(v_i) = f(v_j)$ for some $i \neq j$, but then the restriction of
    $f$ to the 1-simplex $\{v_i, v_j\}$ is not injective.\qedhere
  \end{itemize}
\end{proof}
The following theorem generalises the ``colouring lemma'' of Hatcher
and Wahl (\cite[Lemma 3.1]{HW}), which is the special case where $X$ is
a simplex.  The proof given below is an adaptation of theirs.
\begin{theorem}\label{thm:simplex-wise-injective}
  Let $X$ be a simplicial complex and $f: \partial I^n \to |X|$ be a
  map which is simplicial with respect to some PL triangulation of
  $\partial I^n$.  Then, if $wCM(X) \geq n$, the triangulation extends
  to a PL triangulation of $I^n$, and $f$ extends to a simplicial map
  $g: I^n \to |X|$ with the property that $g(\Lk(v)) \subset
  \Lk(g(v))$ for each interior vertex $v \in I^n - \partial I^n$.  In
  particular, $g$ is simplexwise injective if $f$ is.
\end{theorem}
\begin{proof}
  Since $|X|$ is in particular $(n-1)$-connected, we may extend $f$ to
  a continuous map $I^n \to |X|$ which, by the simplicial
  approximation theorem, may be assumed simplicial with respect to
  some PL triangulation of $I^n$ extending the given triangulation on
  $\partial I^n$.  Thus there is a PL homeomorphism $I^n \approx |K|$
  for which the extension $h: I^n \to |X|$ is simplicial.  Let us say
  that a simplex $\sigma < K$ is \emph{bad} if any vertex $v \in
  \sigma$ is contained in a 1-simplex $\{v,v'\} \subset \sigma$ with $h(v) =
  h(v')$.  We will describe a procedure which replaces the simplicial
  map $h: I^n \to X$ by a ``better'' one, by changing both the map $h$
  and the simplicial complex $K$, arriving at the desired map $g$ in
  finitely many steps.

  If all bad simplices are contained in $\partial I^n$, we are done.
  If not, let $\sigma < K$ be a bad simplex not contained in $\partial
  I^n$, of maximal dimension $p$.  Then $p > 0$, and we must have
  $h(\Lk(\sigma)) \subset \Lk(h(\sigma))$, since otherwise we could
  join a simplex in $\Lk(\sigma)$ to $\sigma$ and get a bad simplex of
  larger dimension.  Now $|\sigma| \subset |K| \approx I^n$, so $h$
  restricts to a map
  \begin{equation*}
    \partial I^{n-p} \approx \Lk(\sigma) \lra \Lk(h(\sigma)).
  \end{equation*}
  The image $h(\sigma)$ is a simplex of dimension $\leq p-1$, since
  otherwise $h|_\sigma$ would be injective (in fact it has dimension
  $\leq (p-1)/2$ by badness), so $\Lk(h(\sigma))$ has
  \begin{equation*}
    wCM(\Lk(h(\sigma))) \geq n - (p-1) - 1 = n - p,
  \end{equation*}
  and in particular $\Lk(h(\sigma))$ is $(n-p-1)$-connected, so
  $h|_{\Lk(\sigma)}$ extends to a PL map
  \begin{equation*}
    I^{n-p} \approx C(\Lk(\sigma)) \overset{\tilde h}\lra \Lk(h(\sigma)).
  \end{equation*}
  By induction on $n$, we may assume that $\tilde h$ is simplicial
  with respect to a PL triangulation of $C(\Lk(\sigma))$ which extends
  the triangulation of $\Lk(\sigma)$, and such that all bad simplices
  of $\tilde h$ are in $\partial I^{n-p} = \Lk(\sigma)$.  We may
  extend this by joining with $h|_{\partial \sigma}$ to get a map
  \begin{equation*}
    \sigma \ast \Lk(\sigma) \approx (\partial \sigma) \ast (C
    \Lk(\sigma)) \overset{\tilde h}\lra X
  \end{equation*}
  which we may finally extend to $I^n$ by setting it equal to $h$
  outside $\sigma \ast \Lk(\sigma) \subset |K|$.  The new map
  $\tilde h$ has fewer bad simplices (not contained in $\partial I^n$)
  of dimension $p$.
\end{proof}

\subsection{Serre microfibrations}
\label{sec:serre-micr}

Let us recall from \cite{WeissCatClass} that a map $p: E \to B$ is called a
\emph{Serre microfibration} if for any $k$ and any lifting diagram
\begin{equation*}
  \xymatrix{
    \{0\} \times D^k \ar[r]^-f \ar[d] & E \ar[d]^p\\
    [0,1] \times D^k \ar[r]^-h & B}
\end{equation*}
there exists an $\epsilon > 0$ and a map $H: [0,\epsilon]\times D^k
\to E$ with $H(0,x) = f(x)$ and $p \circ H(t,x) = h(t,x)$ for all $x
\in D^k$ and $t \in [0,\epsilon]$.  This condition implies that if
$(X,A)$ is a finite CW pair then any map $X \to B$ may be lifted in a
neighbourhood of $A$, extending any prescribed lift over $A$. It also
implies the following useful observation: suppose $(Y,X)$ is a finite CW
pair and we are given a lifting problem
\begin{equation}\label{eq:1}
    \begin{aligned}
      \xymatrix{
        X \ar[r]^-f \ar[d] & E \ar[d]^p\\
        Y \ar[r]^-F & B.}
    \end{aligned}
\end{equation}
If there exists a map $G : Y \to E$ lifting $F$ and so that
$G\vert_X$ is fibrewise homotopic to $f$, then there is also a lift $H$ of
$F$ so that $H\vert_X=f$. To see this, choose a fibrewise homotopy
$\varphi : [0,1] \times X \to E$ from $G\vert_X$ to $f$, let $J =
([0,1] \times X) \cup (\{0\} \times Y) \subset [0,1] \times Y$ and
write $\varphi \cup G : J \to E$ for the map induced by $\varphi$ and
$G$. The following diagram is then commutative
  \begin{equation*}
    \xymatrix{
      J \ar[rr]^-{\varphi \cup G} \ar[d] & &  E \ar[d]^p\\
      [0,1] \times Y \ar[r]^-{\pi_Y}  & Y \ar[r]^-{F} & B, }
  \end{equation*}
and by the microfibration property there is a lift $g : U \to E$ defined on an open neighbourhood $U$ of $J$. Let $\phi : Y \to [0,1]$ be a continuous function with graph inside $U$ and so that $X \subset \phi^{-1}(1)$. Then we set $H(y) = g(\phi(y),y)$; this is a lift of $F$ as $g$ is a lift of $F \circ \pi_Y$, and if $y \in X$ then $\phi(y)=1$ and so $H(y) = g(1,y) = f(y)$, as required.

Examples of Serre microfibrations include submersions of manifolds,
and when $E$ is an open subspace of the total space of a Serre
fibration (more generally, an open subset of another Serre
microfibration).  Weiss proved in \cite[Lemma 2.2]{WeissCatClass} that
if $f: E \to B$ is a Serre microfibration with weakly contractible
fibres (i.e.\ $f^{-1}(b)$ is weakly contractible for all $b \in B$),
then $f$ is in fact a Serre fibration and hence a weak equivalence.
We shall need the following generalisation, whose proof is essentially
the same as Weiss'.
\begin{proposition}\label{prop:Weiss-lemma}
  Let $p: E \to B$ be a Serre microfibration such that $p^{-1}(b)$ is
  $n$-connected for all $b \in B$.  Then the homotopy fibres of $p$
  are also $n$-connected, i.e.\ $p$ is $(n+1)$-connected.
\end{proposition}
\begin{proof}
  Let us first prove that $p^I: E^I \to B^I$ is a Serre microfibration
  with $(n-1)$-connected fibres, where $X^I = \mathrm{Map}([0,1],X)$
  is the space of (unbased) paths in $X$, equipped with the
  compact-open topology.  Using the mapping space adjunction, it is
  obvious that $p^I$ is a Serre microfibration, and showing the
  connectivity of its fibres amounts to proving that any diagram of
  the form
  \begin{equation*}
    \xymatrix{
       [0,1] \times \partial D^k \ar[rr] \ar[d] & & E\ar[d]^p\\
      [0,1] \times D^k \ar[r]^-{\text{proj}}& [0,1] \ar[r] & B }
  \end{equation*}
  with $k \leq n$ admits a diagonal $h: [0,1] \times D^k \to E$.
  Since fibres of $p$ are $(k-1)$-connected (in fact $k$-connected),
  such a diagonal can be found on each $\{a\} \times D^k$, and by the
  microfibration property these lifts extend to a neighbourhood.  By
  the Lebesgue number lemma we may therefore find an integer $N \gg 0$
  and lifts $h_i: [(i-1)/N,i/N]\times D^k \to E$ for $i = 1, \dots,
  N$.  The two restrictions $h_i, h_{i+1}: \{i/N\} \times D^k \to E$
  agree on $\{i/N\} \times \partial D^k$ and map into the same fibre
  of $p$.  Since these fibres are $k$-connected, the restrictions of
  $h_i$ and $h_{i+1}$ are homotopic relative to $\{i/N\}
  \times \partial D^k$ as maps into the fibre, and we may use
  diagram~\eqref{eq:1} with $Y = [i/N,(i+1)/N] \times D^k$ and $X =
  (\{i/N\}\times D^k) \cup ([i/N,(i+1)/N] \times \partial D^k)$ to
  inductively replace $h_{i+1}$ with a homotopy which can be
  concatenated with $h_i$.  The concatenation of the $h_i$'s then
  gives the required diagonal.

  Let us now prove that for all $k \leq n$, any lifting diagram
  \begin{align*}
    \begin{aligned}
      \xymatrix{
        \{0\} \times I^k \ar[r]^-f \ar[d] & E \ar[d]^p\\
        [0,1] \times I^k \ar@{..>}[ur]^H\ar[r]^-h & B }
    \end{aligned}
  \end{align*}
  admits a diagonal map $H$ making the diagram commutative.  To see
  this, we first use that fibres of the map $p^{I^{k+1}}: E^{I^{k+1}}
  \to B^{I^{k+1}}$ are non-empty (in fact $(n-k-1)$-connected) to find
  a diagonal $G$ making the lower triangle commute.  The restriction
  of $G$ to $\{0\} \times I^k$ need not agree with $f$, but they lie
  in the same fiber of $p^{I^k}: E^{I^k} \to B^{I^k}$.  Since this map
  has path connected fibres, these are fibrewise homotopic, and hence
  we may apply~\eqref{eq:1} to replace $G$ with a lift $H$ making both
  triangles commute.

  This homotopy lifting property implies that the inclusion of
  $p^{-1}(b)$ into the homotopy fibre of $p$ over $b$ is
  $n$-connected, and hence that the homotopy fibre is $n$-connected.
\end{proof}

\subsection{Semisimplicial sets and spaces}
\label{sec:connectivity}

Let $\Delta_\mathrm{inj}^*$ be the category whose objects are the
ordered sets $[p] = (0 < \dots < p)$ with $p \geq -1$, and whose
morphisms are the injective, order preserving functions.  An augmented
semisimplicial set is a contravariant functor $X$ from
$\Delta_\mathrm{inj}^*$ to the category of sets.  As usual, such a
functor is specified by the sets $X_p = X([p])$ and face maps $d_i:
X_p \to X_{p-1}$ for $i = 0, \dots, p$.  A (non-augmented)
semisimplicial set is a functor defined on the full subcategory
$\Delta_\mathrm{inj}$ on the objects with $p \geq 0$.  Semisimplicial
spaces are defined similarly.  We shall use the following well known
result.
\begin{proposition}
  \label{prop:connectivity-of-realisation}
  Let $f_\bullet : X_\bullet \to Y_\bullet$ be a map of semisimplicial
  spaces such that $f_p: X_p \to Y_p$ is $(n-p)$-connected for all
  $p$.  Then $|f_\bullet|: |X_\bullet| \to |Y_\bullet|$ is
  $n$-connected.\qed
\end{proposition}

Let us briefly discuss the relationship between simplicial complexes
and semisimplicial sets.  To any simplicial complex $K$ there is an
associated semisimplicial set $K_\bullet$, whose $p$-simplices are the
injective simplicial maps $\Delta^p \to K$, i.e.\ ordered
$(p+1)$-tuples of vertices in $K$ spanning a $p$-simplex.  There is a
natural surjection $|K_\bullet| \to |K|$, and any choice of total
order on the set of vertices of $K$ induces a splitting $|K| \to
|K_\bullet|$.  In particular, $|K|$ is at least as connected as
$|K_\bullet|$.

\begin{proposition}\label{prop:semisimplicial-Serre-mic-fib}
  Let $Y_\bullet$ be a semisimplicial set, and $Z$ be a Hausdorff
  space. Let $X_\bullet \subset Y_\bullet \times Z$ be a
  sub-semisimplicial space which in each degree is an open
  subset. Then $\pi: |X_\bullet| \to Z$ is a Serre microfibration.
\end{proposition}
\begin{proof}
  For $\sigma \in Y_n$, let us write $Z_\sigma \subset Z$ for the open
  subset defined by $(\{\sigma\} \times Z) \cap X_n = \{\sigma\}
  \times Z_\sigma$.  Points in $|X_\bullet|$ are described by data
$$(\sigma \in Y_n ;\, z \in Z_\sigma;\, (t_0, \ldots, t_n) \in \Delta^n)$$
up to the evident relation when some $t_i$ is zero, but we emphasise
that the continuous, injective map $\iota = p \times \pi : |X_\bullet|
\hookrightarrow |Y_\bullet| \times Z$ will not typically be a
homeomorphism onto its image.

Suppose we have a lifting problem
  \begin{equation*}
    \xymatrix{
       \{0\} \times D^k \ar[r]^-{f} \ar[d] &  \vert X_\bullet \vert\ar[d]^{\pi}\\
      [0,1] \times D^k \ar[r]^-{F} & Z. }
  \end{equation*}
  The composition $D^k \overset{f}\to |X_\bullet| \overset{p}\to
  |Y_\bullet|$ is continuous, so the image of $D^k$ is compact and
  hence contained in a finite subcomplex, and it intersects finitely
  many open simplices $\{\sigma_i\} \times \Int(\Delta^{n_i})\subset
  |Y_\bullet|$.  The sets $C_{\sigma_i} = (p \circ f)^{-1}(\{\sigma_i\}
  \times \Int(\Delta^{n_i}))$ then cover $D^k$, and their closures
  $\overline{C}_{\sigma_i}$ give a finite cover of $D^k$ by closed
  sets. Let us write $f\vert_{{C}_{\sigma_i}}(x) =
  (\sigma_i;z(x);t(x))$, with $z(x) \in Z_{\sigma_i} \subset Z$ and
  $t(x) = (t_0(x), \ldots, t_{n_i}(x)) \in \Int (\Delta^{n_i})$.

Certainly $\pi \circ f$ sends the set $C_{\sigma_i}$ into the open set
$Z_{\sigma_i}$, but we claim that $\overline{C}_{\sigma_i}$ is also
mapped into $Z_{\sigma_i}$.  To see this, we consider a sequence
$(x^j) \in C_{\sigma_i}$, $j \in \bN$ converging to a point $x \in
\overline{C}_{\sigma_i} \subset D^k$ and verify that $z= \pi\circ f(x)
\in Z_{\sigma_i}$.  As $f$ is continuous, the sequence $f(x^j) =
(\sigma_i;z(x^j);t(x^j)) \in |X_\bullet|$ converges to $f(x)$, and
passing to a subsequence, we may assume that the $t(x^j)$ converge to
a point $t\in\Delta^{n_i}$.  The
subset
\begin{equation*}
  A = \{f(x^j)\,\, \vert\,\, j \in \bN\} \subset |X_\bullet|
\end{equation*}
is contained in $\pi^{-1}(Z_{\sigma_i})$ and has $f(x)$ as a limit
point in $|X_\bullet|$, so if $z = \pi(f(x)) \not\in Z_{\sigma_i}$,
the set $A$ is not closed in $|X_\bullet|$.  For a contradiction, we
will show that $A$ is closed, by proving that its inverse image in
$\coprod_\tau \{\tau\} \times Z_\tau \times \Delta^{|\tau|}$ is
closed, where the coproduct is over all simplices $\tau \in \coprod_n
Y_n$.  The inverse image in $\{\sigma_i\} \times Z_{\sigma_i} \times
\Delta^{n_i}$ is
\begin{equation*}
  B = \{(\sigma_i;\, z(x^j);\, t(x^j))\,\, \vert\,\, j \in \bN\},
\end{equation*}
which is closed (since $Z$ is Hausdorff, taking the closure in
$\{\sigma_i\} \times Z \times \Delta^{n_i}$ adjoins only the point
$(\sigma_i; z; t)$, which by assumption is outside $\{\sigma_i\} \times
Z_{\sigma_i} \times \Delta^{n_i}$).  If $\sigma_i = \theta^*(\tau)$
for a morphism $\theta \in \Delta_{\mathrm{inj}}$, we have $Z_\tau
\subset Z_{\sigma_i}$ and hence $B \cap (\{\sigma_i\} \times Z_\tau
\times \Delta^{|\sigma_i|})$ is closed in $\{\sigma_i\} \times Z_\tau
\times \Delta^{|\sigma_i|}$ so applying $\theta_*: \Delta^{|\sigma_i|}
\to \Delta^{|\tau|}$ gives a closed subset $B_\theta \subset
\{\tau\} \times Z_\tau \times \Delta^{|\tau|}$.  The inverse image of
$A$ in $\{\tau\} \times Z_\tau \times \Delta^{|\tau|}$ is the union of
the $B_\theta$ over the finitely many $\theta$ with $\theta^*(\tau) =
\sigma_i$, and is hence closed.

We have a continuous map $F_i = F\vert_{[0,1] \times \overline{C}_{\sigma_i}} : [0,1] \times \overline{C}_{\sigma_i} \to
Z$ and $F_i^{-1}(Z_{\sigma_i})$ is an open neighbourhood of the
compact set $\{0\} \times \overline{C}_{\sigma_i}$, so there is an
$\epsilon_i > 0$ such that $F_i([0,\epsilon_i] \times
\overline{C}_{\sigma_i}) \subset Z_{\sigma_i}$. We set $\epsilon =
\min_i (\epsilon_i)$ and define the lift
$$\widetilde{F}_i(s, x) = (\sigma_i;F_i(s,x);t(x)) : [0,\epsilon]
\times \overline{C}_{\sigma_i} \lra \{\sigma_i\} \times Z_{\sigma_i} \times \Delta^{n_i} \lra |X_\bullet|,$$
which is clearly continuous. The functions $\widetilde{F}_i$ and
$\widetilde{F}_j$ agree where they are both defined, and so these glue
to give a continuous lift $\widetilde{F}$ as required.
\end{proof}

\begin{corollary}\label{cor:serre-microf-connectivity}
  Let $Z$, $Y_\bullet$, and $X_\bullet$ be as in Proposition
  \ref{prop:semisimplicial-Serre-mic-fib}.  For $z \in Z$, let
  $X_\bullet(z) \subset Y_\bullet$ be the sub-semisimplicial set
  defined by $X_\bullet \cap (Y_\bullet \times \{z\}) = X_\bullet(z)
  \times \{z\}$ and suppose that $|X_\bullet(z)|$ is $n$-connected for all
  $z \in Z$.  Then the map $\pi: |X_\bullet| \to Z$ is
  $(n+1)$-connected.
\end{corollary}
\begin{proof}
  This follows by combining Propositions~\ref{prop:Weiss-lemma}
  and~\ref{prop:semisimplicial-Serre-mic-fib}, once we prove that
  $|X_\bullet(z)|$ is homeomorphic to $\pi^{-1}(z)$ (in the subspace
  topology from $|X_\bullet|$).  Since $X_\bullet(z)\subset
  Y_\bullet$, the composition $|X_\bullet(z)| \to |X_\bullet| \to
  |Y_\bullet|$ is a homeomorphism onto its image.  It follows that
  $|X_\bullet(z)| \to |X_\bullet|$ is a homeomorphism onto its image,
  which is easily seen to be $\pi^{-1}(z)$.
\end{proof}

%%% Local Variables: 
%%% mode: latex
%%% TeX-master: "stability"
%%% End: 

%% file: chap3.tex
\section{Algebra}\label{sec:algebra}

We fix $\epsilon = \pm 1$. Let $\Lambda \subset \bZ$ be a subgroup satisfying
$$\{a - \epsilon {a} \,\,\vert\,\, a \in \bZ \} \subset \Lambda \subset \{a \in \bZ \,\,\vert\,\, a + \epsilon {a} =0\}.$$
Following Bak (\cite{Bak, Bak2}), we call such a pair $(\epsilon,
\Lambda)$ a \emph{form parameter}. An \emph{$(\epsilon,
  \Lambda)$-Quadratic module} $(M, \lambda, \alpha)$ is the data of a
$\bZ$-module $M$, a $\epsilon$-symmetric bilinear form
$$\lambda : M \otimes M \lra \bZ$$
and a $\Lambda$-quadratic form
$$\alpha : M \lra \bZ/\Lambda$$
whose associated bilinear form is $\lambda$ reduced modulo
$\Lambda$. By this we mean a function $\alpha$ such that
\begin{enumerate}[(i)]
\item $\alpha(a \cdot x) = a^2 \cdot \alpha(x)$ for $a \in \bZ$,
\item $\alpha(x+y) = \alpha(x) + \alpha(y) + \lambda (x, y)$.
\end{enumerate}
We say the $(\epsilon, \Lambda)$-Quadratic module is
\emph{non-degenerate} if the map
\begin{align*}
  M &\lra M^*\\
  x &\longmapsto \lambda(-, x)
\end{align*}
is an isomorphism. A \emph{morphism} of $(\epsilon,
\Lambda)$-Quadratic modules is a homomorphism $f : M \to N$ of modules
which is an isometry for $\lambda$, and such that $\alpha_M = \alpha_N
\circ f$. If $M$ is non-degenerate, any such morphism is injective, as
$$M \overset{f}\lra N \lra N^* \overset{f^*}\lra M^*$$
is an isomorphism. The \emph{hyperbolic module} $H$ is the $(\epsilon,
\Lambda)$-Quadratic module given by the data
$$\left(\bZ^2 \text{ with basis } e, f ;\, 
\left (\begin{array}{cc}
0 & 1 \\
\epsilon & 0
\end{array}\right);\, \alpha(e)=\alpha(f)=0\right).$$

\begin{definition}
  For an $(\epsilon, \Lambda)$-Quadratic module $(M,\lambda,\alpha)$,
  let $K^a(M)$ be the simplicial complex whose vertices are morphisms
  $e:H \to M$ of quadratic modules.  The set $\{e_0, \dots, e_p\}$ is
  a $p$-simplex if the submodules $e_i(H) \subset M$ are orthogonal
  with respect to $\lambda$ (and no condition on the quadratic forms).
 
  If $\sigma = \{v_0, \ldots, v_p\} < K^a(M)$, then the link
  $\Lk(\sigma)$ is isomorphic to $K^a(M \cap \gen(v_0, \ldots,
  v_p)^\perp)$.
\end{definition}

This complex is almost the same as one considered by Charney, which
she proves to be highly connected.
\begin{theorem}\label{thm:Charney}
  Let $M = H^{\oplus g}$. Then $|K^a(M)|$ is $\lfloor (g-5)/2 \rfloor$-connected.
\end{theorem}
\begin{proof}
  The simplicial complex $K^a(M)$ has an associated semisimplicial set
  $K^a_\bullet(M)$, and in {\cite[Corollary 3.3]{Charney}} it is
  proved that (the barycentric subdivision of) $|K^a_\bullet(M)|$ is
  $\lfloor (g-5)/2\rfloor$-connected.  As mentioned in Section
  \ref{sec:connectivity}, this implies that $|K^a(M)|$ is also
  $\lfloor (g-5)/2\rfloor$-connected.
\end{proof}

\begin{corollary}[Transitivity]\label{cor:AlgebraicTransitivity}
  If $e_0, e_1 : H \to H^{\oplus g}$ are morphisms of quadratic modules
  and $g \geq 5$, there is an isomorphism of quadratic modules $f :
  H^{\oplus g} \to H^{\oplus g}$ such that $e_1 = f \circ e_0$.
\end{corollary}
\begin{proof}
Suppose first that $e_0$ and $e_1$ are orthogonal. Then
$$H^{\oplus g} \cong e_0(H) \oplus e_1(H) \oplus M$$
and there is an evident automorphism of quadratic modules which swaps the $e_i(H)$. Now, the relation between morphisms $e : H \to H^{\oplus g}$ of differing by an automorphism is an equivalence relation, and we have just shown that adjacent vertices in $K^a(H^{\oplus g})$ are equivalent. When $g \geq 5$ this simplicial complex is connected, and so all vertices are equivalent.
\end{proof}

\begin{corollary}[Cancellation]\label{cor:AlgebraicCancellation}
Suppose that $M$ is a quadratic module and there is an isomorphism $ M \oplus H \cong H^{\oplus g+1}$ for $g \geq 4$. Then $M \cong H^{\oplus g}$.
\end{corollary}
\begin{proof}
An isomorphism $\varphi : M \oplus H \to H^{\oplus g+1}$ gives a morphism $\varphi\vert_H : H \to H^{\oplus g+1}$ of quadratic modules, and we also have the standard inclusion $e_{g+1} : H \to H^{\oplus g+1}$. When $g+1 \geq 5$, the previous corollary shows that these differ by an automorphism of $H^{\oplus g+1}$, and in particular their orthogonal complements are isomorphic.
\end{proof}

\begin{corollary}\label{cor:wCM-Ka}
  Let $M = H^{\oplus g}$. Then  $wCM(K^a(M)) \geq \lfloor (g-3)/2\rfloor$.
\end{corollary}
\begin{proof}
  We prove the statement assuming $g \geq 4$ (it is obvious for $g
  \leq 4$).  Let $\sigma = \{ v_0, \ldots, v_p \} < K^a(H^{\oplus g})$
  be a $p$-simplex, with link $K^a(\gen(e_0, \ldots, e_p)^\perp)$. As
  $$H^{\oplus g} \cong e_0(H) \oplus \cdots \oplus e_p(H) \oplus \gen(e_0, \ldots, e_p)^\perp$$
  by Corollary \ref{cor:AlgebraicCancellation}, $K^a(\gen(e_0, \ldots,
  e_p)^\perp)$ is isomorphic to the complex $K^a(H^{\oplus g-p-1})$,
  as long as $g-p-1 \geq 4$. In this case the link is then
  $\lfloor(g-p-6)/2\rfloor$-connected, but
  \begin{equation*}
    \lfloor (g-p-6)/2\rfloor \geq \lfloor (g-3)/2\rfloor - p-2,
  \end{equation*}
  which proves the claim.  In the case $g-p-1 <4$ we have $\lfloor
  (g-3)/2\rfloor -p-2 \leq \lfloor (1-g)/2\rfloor \leq -2$, so there
  is no condition on the link in this case.
\end{proof}
%%% Local Variables: 
%%% mode: latex
%%% TeX-master: "stability"
%%% End: 

%% file: chap4.tex
\section{Topology}

We fix a dimension $d=2n \geq 6$ throughout and write
$$W_g = \#^g S^n \times S^n$$
for the $g$-fold connected sum of $S^n \times S^n$ with itself.  This
is a $2n$-dimensional smooth closed manifold, and we write $W_{g,k}$
for the manifold with boundary obtained by removing the interiors of
$k$ disjoint discs from $W_g$.  (Up to diffeomorphism, the manifold
$W_{g,k}$ does not depend on choices of where to perform connected sum
and which discs to cut out.  We make these choices once and for all,
and the notation $W_{g,k}$ will denote an actual abstract manifold,
with boundary components parametrised by $S^{2n-1}$, rather than a
diffeomorphism class.)  It will be convenient to have available the
following small modification of the manifold $W_{1,1}$.  Let $H$
denote the manifold obtained from $W_{1,1} = S^n \times S^n
-\Int(D^{2n})$ by gluing $[0,1] \times D^{2n-1}$ onto $\partial
W_{1,1}$ along an oriented embedding
\begin{equation*}
  \{1\} \times D^{2n-1} \lra \partial W_{1,1},
\end{equation*}
which we also choose once and for all.  This gluing of course doesn't
change the diffeomorphism type (after smoothing corners), so $H$ is
diffeomorphic to $W_{1,1}$, but contains a standard embedded $[0,1]
\times D^{2n-1} \subset H$. When we discuss embeddings of $H$ into a
manifold with boundary $W$, we shall always insist that $\{0\} \times
D^{2n-1}$ is sent into $\partial W$, and that the rest of $H$ is sent
into the interior of $W$.

We shall also need a \emph{core} $C \subset H$, defined as follows.
Let $x_0 \in S^n$ be a basepoint.  Let $S^n \vee S^n = (S^n \times
\{x_0\}) \cup (\{x_0\} \times S^n) \subset S^n \times S^n$, which we
may suppose is contained in $\Int(W_{1,1})$.  Choose an embedded path
$\gamma$ in $\Int(H)$ from $(x_0, -x_0)$ to $(0,0) \in [0,1] \times
D^{2n-1}$ whose interior does not intersect $S^n \vee S^n$, and whose
image agrees with $[0,1] \times \{0\}$ inside $[0,1] \times D^{2n-1}$,
and let
\begin{equation*}
  C = (S^n \vee S^n) \cup \gamma([0,1]) \cup (\{0\} \times D^{2n-1})
  \subset H.
\end{equation*}
We may choose an isotopy of embeddings $\rho_t : H \to H$, defined for
$t \in [0,\infty)$, which starts at the identity, eventually has image
inside any given neighbourhood of $C$, and which for each $t$ is the
identity on some neighbourhood of $C$.

\begin{definition}\label{defn:K-p}
  Let $W$ be a compact manifold, equipped with a fixed embedding $c:
  [0,1) \times \R^{2n-1} \to W$ such that $c^{-1}(\partial W) = \{0\}
  \times \R^{2n-1}$.
  \begin{enumerate}[(i)]
  \item\label{item:5} Let $K_0(W) = K_0(W,c)$ be the space of pairs $(t,\phi)$,
    where $t \in \R$ and $\phi: H \to W$ is an embedding whose
    restriction to $[0,1) \times D^{2n-1} \subset H$ satisfies
    that there exists an $\epsilon > 0$ such that
    \begin{equation*}
      \phi(s,p) = c(s, p + te_1)
    \end{equation*}
    for all $s < \epsilon$ and all $p \in D^{2n-1}$.  Here, $e_1 \in
    \R^{2n-1}$ denotes the first basis vector.
  \item Let $K_p(W) \subset (K_0(W))^{p+1}$ consist of those tuples
    $((t_0, \phi_0), \dots, (t_p, \phi_p))$ satisfying that $t_0 <
    \dots < t_p$ and that the embeddings $\phi_i$ have disjoint cores, i.e.\ the sets $\phi_i(C)$ are disjoint.
  \item Topologise $K_p(W)$ using the $C^\infty$-topology on the space
    of embeddings and let $K_p^\delta(W)$ be the same set considered
    as a discrete topological space.
  \item The assignments $[p] \mapsto K_p(W)$ and $[p] \mapsto
    K_p^\delta(W)$ define semisimplicial spaces, where the face map
    $d_i$ forgets $(t_i,\phi_i)$.
  \item Let $K^\delta(W)$ be the simplicial complex with vertices
    $K^\delta_0(W)$, and where the (unordered) set $\{(t_0,\phi_0),
    \dots, (t_p,\phi_p)\}$ is a $p$-simplex if, when written with $t_0
    < \dots < t_p$, it satisfies $((t_0, \phi_0), \dots, (t_p,
    \phi_p)) \in K_p^\delta(W)$.
  \end{enumerate}
  We shall often denote a vertex $(t,\phi)$ simply by $\phi$, since
  $t$ is determined by $\phi$.  Since a $p$-simplex of
  $K_\bullet^\delta(W)$ is determined by its (unordered) set of
  vertices, there is a natural homeomorphism $|K_\bullet^\delta(W)| =
  |K^\delta(W)|$.  % The identity function induces a semisimplicial map
  % $\iota: K_\bullet^\delta(W) \to K_\bullet(W)$.
\end{definition}

The fibration $S^{n} \to BO(n) \to BO(n+1)$ gives an exact sequence
$$\cdots \lra \pi_{n+1}(BO(n+1)) \overset{\partial}\lra \pi_{n}(S^{n}) = \bZ \overset{\tau}\lra \pi_{n}(BO(n)) \overset{s}\lra \pi_{n}(BO(n+1)) \lra 0$$
and we define $\Lambda_n := \IM(\partial) \subset \bZ$. This can of course be explicitly computed: it is 0 if $n$ is even (the Euler class detects the injectivity of $\tau$), $\bZ$ if $n\in \{1, 3, 7\}$, and $2\bZ$ otherwise, by the Hopf invariant 1 theorem. 

The data $((-1)^n, \Lambda_n)$ is a \emph{form parameter}, in the
sense of Section \ref{sec:algebra}. Following Wall (\cite{Wall62}), we
will now construct from a stably parallelisable, $(n-1)$-connected
$2n$-manifold $W$, a quadratic module having this form parameter. The
first non-zero homotopy group of such a manifold is
$$\pi_{n}(W) \cong H_{n}(W;\bZ).$$
Using the intersection form on the middle homology of $W$, we obtain a bilinear form
$$\lambda : \pi_n(W) \otimes \pi_n(W) \lra \bZ$$
which is $(-1)^n$-symmetric, and non-degenerate by Poincar{\'e}
duality. By a theorem of Haefliger (\cite{Haefliger}), an element $x \in \pi_n(W)$ may be represented by an embedded sphere as long as $n \geq 3$, and this representation is unique up to isotopy as long as $n \geq 4$. Such an embedding has a normal bundle which is stably trivial, so is represented by an element
$$\alpha(x) \in \Ker\left(\pi_n(BO(n)) \overset{s}\to \pi_n(BO)\right) = \bZ/\Lambda_n.$$
This gives a well-defined function $\alpha : \pi_n(W) \to \bZ/\Lambda_n$ (it is well-defined even when $n=3$, as $\bZ/\Lambda_3=\{0\}$).

\begin{lemma}
The data $(\pi_n(W), \lambda, \alpha)$ is a $((-1)^n, \Lambda_n)$-Quadratic module.
\end{lemma}
\begin{proof}
See \cite[Lemma 2]{Wall62}.
\end{proof}

If we start with the manifold $W_{g,1} = \#^g S^n \times S^n \setminus
\Int(D^{2n})$ then the associated quadratic module is isomorphic to $H^{\oplus g}$. Furthermore, given an embedding $e : H
\to W_{g,1}$ we have an induced morphism $E : H \to \pi_n(W_{g,1})$ of
quadratic modules, and disjoint embeddings give orthogonal
morphisms. This defines a map of simplicial complexes
\begin{equation}
  \label{eq:7}
  K^\delta(W_{g,1}) \lra K^a(\pi_n(W_{g,1}), \lambda,\alpha)
\end{equation}
which we will use to compute the connectivity of $\vert
K_\bullet^\delta(W_{g,1})\vert = |K^\delta(W_{g,1})|$.  For brevity we
shall just write $K^\delta \to K^a$ for this map in the proof of the
following result.

\begin{lemma}\label{lemthm:conn-K-delta}
  The space $|K^\delta_\bullet(W_{g,1})|$ is $\lfloor (g-5)/2
  \rfloor$-connected.
\end{lemma}
\begin{proof}
  Let $k \leq (g-5)/2$ and consider a map $f: S^k \to |K^\delta|$,
  which we may assume is simplicial with respect to some PL
  triangulation of $S^{k} = \partial I^{k+1}$.  By
  Theorem~\ref{thm:Charney}, the composition $\partial I^{k+1} \to
  |K^\delta| \to |K^a|$ is null-homotopic and so extends to a map $g:
  I^{k+1} \to |K^a|$, which we may suppose is simplicial with respect
  to a PL triangulation of $I^{k+1}$ extending the triangulation of
  its boundary.
  
  By Corollary \ref{cor:wCM-Ka}, we have $wCM(K^a) \geq \lfloor
  (g-3)/2\rfloor$. By Theorem~\ref{thm:simplex-wise-injective}, as
  $k+1 \leq \lfloor (g-3)/2 \rfloor$ we can arrange that $g$ is
  simplexwise injective on the interior of $I^{k+1}$.  Then we pick a
  total order on the interior vertices, and inductively pick lifts of
  each vertex to $K^\delta_0$.  In each step, a vertex is given by a
  morphism of quadratic modules $J: H \to \pi_n(W_{g,1})$. The element $J(e)$ is
  represented by a map $x : S^n \to W_{g,1}$ which by Haefliger's
  theorem is representable by an embedding, and as $\alpha(x) =
  \alpha(e)=0$ this embedding has trivial normal bundle. Thus $J(e)$
  can be represented by an embedding $j(e) : S^n \times D^n \to
  W_{g,1}$. Similarly, $J(f)$ can be represented by an embedding $j(f)
  : S^n \times D^n \to W_{g,1}$.

  As $\lambda(J(e), J(f))=1$, these two embeddings have algebraic
  intersection number 1. As $W_{g,1}$ is simply-connected and of
  dimension at least 6, we may use the Whitney trick to isotope these
  embeddings so that their cores $S^n \times \{0\}$ intersect
  transversely in precisely one point, and so obtain an embedding of
  the plumbing of $S^n \times D^n$ and $D^n \times S^n$, which is
  diffeomorphic to $W_{1,1} \subset H$.  To extend to the remaining
  $[0,1]\times D^{2n-1} \subset H$, we first pick an embedding $\{0\}
  \times D^{2n-1} \to \partial W_{g,1}$ disjoint from previous
  embeddings and satisfying condition~(\ref{item:5}) of
  Definition~\ref{defn:K-p}; then extend to an embedding of $[0,1]
  \times \{0\}$ which we thicken to an embedding of $[0,1] \times
  D^{2n-1}$.  Finally, as $J$ is orthogonal to any adjacent vertices
  which have already been lifted, we can use the Whitney trick again
  to isotope $j$ so that its core is disjoint from the cores of all
  previously chosen vertices that are adjacent to it. After applying
  this procedure to all vertices, we obtain a lift of $g$ to a
  null-homotopy of $f$, as required.
\end{proof}

We now make deductions from the fact that
$|K_\bullet^\delta(W_{g,1})|$ is path connected when $g \geq 5$,
similar to those made about $K^a(H^{\oplus g})$ in Section
\ref{sec:algebra}.

\begin{corollary}[Transitivity]\label{cor:GeometricTransitivity}
Let $e_0, e_1 : H \hookrightarrow W_{g,1}$ be embeddings, and $g \geq 5$. Then there is a diffeomorphism $f$ of $W_{g,1}$ which is isotopic to the identity on the boundary and such that $e_1 = f \circ e_0$.
\end{corollary}
\begin{proof}
  Suppose first that $e_0$ and $e_1$ are disjoint. Let $V$ denote the
  closure of a regular neighbourhood of $e_0(H) \cup e_1(H)
  \cup \partial W_{g,1}$, which is abstractly diffeomorphic to
  $W_{2,2}$ with two standard copies of $H$ embedded in it, both
  connected to the first boundary. It is enough to find a
  diffeomorphism of $W_{2,2}$ which is the identity on the first
  boundary, is isotopic to the identity on the second boundary, and
  sends the first $H$ to the second.

  We give a concrete construction of such a diffeomorphism. Let
  $\Gamma \in SO(2n)$ be the diagonal matrix with entries
  $(-1,-1,1,1,\ldots,1)$. Let $d_0 : D^{2n} \to [0,1] \times S^{2n-1}$
  be a small disc around $(\tfrac{1}{2}, (1,0,\ldots,0))$, so $d_1 =
  (\mathrm{Id}_{[0,1]} \times \Gamma) \cdot d_0$ is a small disc
  around $(\tfrac{1}{2}, (-1,0,\ldots,0))$. We form a manifold by
  connect summing two copies of $S^n \times S^n$ to $[0,1] \times
  S^{2n-1}$ at these two discs,
  \begin{equation*}
    M_{2,2} = (S^n \times S^n) \#_{d_0} ([0,1] \times S^{2n-1}) \#_{d_1}
    (S^n \times S^n).
  \end{equation*}
  There is a diffeomorphism $\varphi$ of $M_{2,2}$ which is given by
  $\mathrm{Id}_{[0,1]} \times \Gamma$ on $[0,1] \times S^{2n-1}
  \setminus (d_0(D^{2n}) \cup d_1(D^{2n}))$, and interchanges the two
  copies of $S^n \times S^n \setminus D^{2n}$. There is an embedded
  copy of $H$ given by the first $S^n \times S^n \setminus D^{2n}$
  along with a thickening of the arc $[\tfrac{1}{2},1] \times
  \{(1,0,\ldots,0)\}$, and $\varphi(H)$ gives another disjoint
  embedded copy of $H$. We have found the required diffeomorphism,
  except that it is not the identity on the boundary $\{0\} \times
  S^{2n-1}$.  To amend this, we replace $\mathrm{Id}_{[0,1]} \times
  \Gamma$ in the above construction with a function of the form $(t,x)
  \mapsto (t,\gamma(t)x)$, where $\gamma : [0,1] \to SO(2n)$ is a path
  which is the identity on $[0,\epsilon]$ and $\Gamma$ on
  $[2\epsilon,1]$ for some small $\epsilon$.

Now suppose that the $e_i$ merely have disjoint cores, and recall the
isotopy of self-embeddings $\rho_t: H \to H$ from before. For $T \gg
0$ the embeddings $e_i \circ \rho_T$ are disjoint, so by the Isotopy
Extension Theorem we find diffeomorphisms $\varphi_i : W_{g,1} \to
W_{g,1}$ (which are isotopic to the identity) such that the embeddings
$\varphi_i \circ e_i$ are disjoint. By the above case we then find a
diffeomorphism $g$ of $W_{g,1}$ such that $(\varphi_1 \circ e_1) = g
\circ (\varphi_0 \circ e_0)$, so $f = \varphi_1^{-1} \circ g \circ
\varphi_0$ gives the diffeomorphism we require.

To prove the general case, when the $e_i$ are not assumed to have
disjoint cores, we use the connectedness of $\vert
K_\bullet^\delta(W_{g,1})\vert$ when $g\geq 5$ and the argument of
Corollary \ref{cor:AlgebraicTransitivity}.
\end{proof}

\begin{corollary}[Cancellation]\label{cor:GeometricCancellation}
Let $M$ be a $2n$-manifold with boundary parametrised by $S^{2n-1}$, and suppose there is a diffeomorphism
$$\varphi : M \# S^n \times S^n \lra W_{g+1, 1}$$
which is the identity on the boundary. Then if $g \geq 4$ there is a diffeomorphism of $M$ with $W_{g, 1}$ which is the identity on the boundary.
\end{corollary}
\begin{proof}
This is completely analogous to Corollary \ref{cor:AlgebraicCancellation}. We have the two embeddings $\varphi\vert_H : H \to W_{g+1,1}$ and $e_{g+1} : H \to W_{g+1,1}$, and by Corollary \ref{cor:GeometricTransitivity} there is a diffeomorphism $f$ of $W_{g+1,1}$ isotopic to the identity on the boundary so that $e_{g+1} = f \circ \varphi\vert_H$. We obtain a diffeomorphism $f  \circ \varphi : M \# S^n \times S^n \lra W_{g+1, 1}$ isotopic to the identity on the boundary and which sends $H$ to $H$ identically. In particular it sends $W_{1,1} \subset H$ to itself identically, so after removing the interior of $W_{1,1}$ gives a diffeomorphism
$$M \setminus D^{2n} \lra W_{g,1} \setminus D^{2n}$$
which is isotopic to the identity on the old boundary, and is equal to the identity on the new boundary. We can then fill in the new boundary with a standard disc, and use isotopy extension to make the diffeomorphism be the identity on the remaining boundary.
\end{proof}

In fact, Kreck (\cite[Theorem D]{Kreck}) proved the above cancellation result for
$g \geq 1$ (and $g \geq 0$ when $n$ is odd), but we shall only use the
weaker result in Corollary~\ref{cor:GeometricCancellation}.

Finally, we compare $|K_\bullet^\delta(W_{g,1})|$ and $|K_\bullet(W_{g,1})|$.  The
bisemisimplicial space in Definition~\ref{defn:D-bullet-bullet} below
will be used to leverage the known connectivity of
$|K_\bullet^\delta(W_{g,1})|$ to prove the following theorem, which is
the main result of this section.
\begin{theorem}\label{thm:high-conn}
  The space $|K_\bullet(W_{g,1})|$ is $\lfloor
  (g-5)/2\rfloor$-connected.
\end{theorem}

\begin{definition}\label{defn:D-bullet-bullet}
  With $W$ and $c$ as in Definition~\ref{defn:K-p}, let $D_{p,q} =
  K_{p+q+1}(W)$, topologised as a subspace of $K_p(W) \times
  K^\delta_q(W)$.  This is a bisemisimplicial space, equipped with
  augmentations
  \begin{align*}
    D_{p,q} &\overset{\epsilon}\lra K_p(W)\\
    D_{p,q} &\overset{\delta}\lra K_q^\delta(W).
  \end{align*}
\end{definition}
\begin{lemma}\label{lem:homotopy-commute}
  Let $\iota: K_\bullet^\delta(W) \to K_\bullet(W)$ denote the
  identity function.  Then
  \begin{equation*}
    |\iota| \circ |\delta| \simeq |\epsilon|: |D_{\bullet,
      \bullet}| \lra |K_\bullet(W)|.
  \end{equation*}
\end{lemma}
\begin{proof}
  For each $p$ and $q$ there is a homotopy
  \begin{align*}
    [0,1] \times \Delta^p \times \Delta^q \times D_{p,q} \lra
    \Delta^{p+q+1} \times K_{p+q+1}(W)\\
    (r,s,t,x,y)) \longmapsto ((rs,(1-r)t),(x,\iota y)),
  \end{align*}
  where we write $(x,y) \in D_{p,q} \subset K_p(W)\times
  K_q^\delta(W)$ and $(x,\iota y) \in K_{p+q+1}(W) \subset K_p(W)
  \times K_q(W)$ and $(rs,(1-r)t) = (rs_0,\dots, rs_p,(1-r)t_0, \dots,
  (1-r)t_q) \in \Delta^{p+q+1}$.  These homotopies glue to a homotopy
  $[0,1] \times |D_{\bullet,\bullet}| \to |K_\bullet(W)|$ which
  starts at $|\iota| \circ |\delta|$ and ends at $|\epsilon|$.
\end{proof}
\begin{proof}[Proof of Theorem \ref{thm:high-conn}]
  % Given a simplex $\sigma = ((t_0,\phi_0), \dots, (t_p,\phi_p)) \in
  % K_p(W_{g,1})$, we shall write $W_\sigma \subset W$ for the
  % complement of the $\phi_i(C)$.  The fibre of the augmentation
  % \begin{equation*}
  %   |D_{p,\bullet}| \lra K_p(W_{g,1})
  % \end{equation*}
  % over $\sigma$ is the full sub-complex $F(\sigma) \subset
  % K^\delta(W_{g,1})$ consisting of those $(t,\phi)$ such that $\phi(C)
  % \subset W_\sigma$ and $t>t_p$. The map of simplicial
  % complexes~\eqref{eq:7} restricts to a map
  % \begin{equation*}
  %   F(\sigma) \lra K^a(\pi_n(W_\sigma), \lambda, \alpha).
  % \end{equation*}
  % By Corollary \ref{cor:AlgebraicCancellation} the target is
  % isomorphic to $K^a(H^{\oplus g-p-1})$ which is $\lfloor
  % (g-p-6)/2\rfloor$-connected by Theorem~\ref{thm:Charney}, and the
  % argument of Lemma \ref{lemthm:conn-K-delta} shows that $F(\sigma)$
  % is too.  (Corollary~\ref{cor:AlgebraicCancellation} only applies
  % when $g-p-1 \geq 4$, but the statement is vacuously true for
  % $g-p\leq 5$.)

  We will apply Corollary~\ref{cor:serre-microf-connectivity} with $Z
  = K_p(W_{g,1})$, $Y_\bullet = K_\bullet^\delta(W_{g,1})$ and
  $X_\bullet = D_{p,\bullet}$.  For $z = ((t_0, \phi_0), \dots, (t_p,
  \phi_p)) \in K_p(W_{g,1})$, we shall write $W_z \subset W$ for the
  complement of the $\phi_i(C)$.  The realisation of the
  semisimplicial subset $X_\bullet(z) \subset Y_\bullet =
  K_\bullet^\delta(W_{g,1})$ is homeomorphic to the full subcomplex
  $F(z) \subset K^\delta(W_{g,1})$ on those $(t,\phi)$ such that
  $\phi(C) \subset W_z$ and $t > t_p$.  The map of simplicial
  complexes~\eqref{eq:7} restricts to a map
  \begin{equation*}
    F(z) \lra K^a(\pi_n(W_z), \lambda, \alpha).
  \end{equation*}
  By Corollary \ref{cor:AlgebraicCancellation} the target is
  isomorphic to $K^a(H^{\oplus g-p-1})$ which is $\lfloor
  (g-p-6)/2\rfloor$-connected by Theorem~\ref{thm:Charney}, and the
  argument of Lemma \ref{lemthm:conn-K-delta} shows that $F(z)$ is
  too.  (Corollary~\ref{cor:AlgebraicCancellation} only applies when
  $g-p-1 \geq 4$, but the statement is vacuously true for $g-p\leq
  5$.)  By Corollary~\ref{cor:serre-microf-connectivity}, the map
  $|\epsilon|: |D_{p,\bullet}| \to K_p(W_{g,1})|$ is $\lfloor
  (g-p-4)/2\rfloor$-connected and since $\lfloor (g-p-4)/2\rfloor \geq
  \lfloor (g-4)/2\rfloor - p$, we deduce by
  Proposition~\ref{prop:connectivity-of-realisation} that the map
  $|D_{\bullet,\bullet}| \to |K_\bullet(W_{g,1})|$ is $\lfloor (g
  -4)/2 \rfloor$-connected.  But up to homotopy it factors through the
  $\lfloor (g-5)/2\rfloor$-connected space
  $|K_\bullet^\delta(W_{g,1})|$, and therefore $|K_\bullet(W_{g,1})|$
  is $\lfloor (g-5)/2\rfloor$-connected too.
\end{proof}

Finally, define the sub-semisimplicial space $\overline{K}_\bullet(W_{g,1})
\subset K_\bullet(W_{g,1})$ whose $p$-simplices are tuples of disjoint
embeddings.  (Recall that in $K_\bullet(W_{g,1})$ we only ask for the
embeddings to have disjoint cores.)

\begin{corollary}
  The space $|\overline{K}_\bullet(W_{g,1})|$ is $\lfloor
  (g-5)/2\rfloor$-connected.
\end{corollary}
\begin{proof}
  Precomposing with the isotopy $\rho_t$, any tuple of embeddings with
  disjoint cores eventually become disjoint.  It follows that the
  inclusion is a levelwise weak equivalence.
\end{proof}

%%% Local Variables: 
%%% mode: latex
%%% TeX-master: "stability"
%%% End: 

%% file: chap5.tex
\section{Resolutions of moduli spaces}

We now use the high connectivity of $|\overline{K}_\bullet(W_{g,1})|$ to prove
Theorem~\ref{thm:main}.

\subsection{A semisimplicial resolution}
\label{sec:semis-resol}

We shall define a semisimplicial resolution of the moduli space
$\MM_g$, meaning a semisimplicial space $X_\bullet$ with an
augmentation $X_\bullet \to \MM_g$ such that the map $|X_\bullet|
\to \MM_g$ is highly connected (see
Proposition~\ref{prop:resolution} below for the precise meaning).

Before defining $X_\bullet$ in Definition~\ref{defn:X-p}, we
recall that the topology of $\MM_g(\R^N)$ is defined by the
homeomorphism
\begin{equation*}
  \MM_g(\R^N) = \Emb^\partial (W_{g,1}, [0,\infty) \times
  \R^N)/\Diff^\partial(W_{g,1}),
\end{equation*}
where $\Emb^\partial$ denotes the space of embeddings with fixed
behaviour near the boundary (in terms of a collar $[0,1)
\times \partial W_{g,1} \to W_{g,1}$) and $\Diff^\partial$ denotes the
space of diffeomorphisms which fix a neighbourhood of the boundary
pointwise.  It is well known (see e.g.\ \cite{MR613004}) that the
quotient map from the embedding space is a principal
$\Diff^\partial(W_{g,1})$-bundle.
\begin{definition}\label{defn:X-p}
  Pick once and for all a coordinate patch $c_0: \R^{2n-1} \to
  S^{2n-1}$.  This choice induces for any $W \in \MM_g$ a germ of an
  embedding $[0,1) \times \R^{2n-1} \to W$ as in Definition
  \ref{defn:K-p}, and we let $X_p(\R^N)$ be the space of pairs
  $(W,\phi)$ where $W \in \MM_g(\R^N)$ and $\phi \in
  \overline{K}_p(W)$, topologised as
  \begin{equation*}
    X_p(\R^N) = (\Emb^\partial(W_{g,1},[0,\infty)\times\R^N) \times
    \overline{K}_p(W_{g,1}))/\Diff^\partial(W_{g,1}).
  \end{equation*}
  This makes $X_\bullet(\R^N)$ into a semisimplicial space augmented over
  $\MM_g(\R^N)$.  By the local triviality of the quotient map defining
  the topology on $\MM_g(\R^N)$, the augmentation $X_\bullet(\R^N)\to
  \MM_g(\R^N)$ is locally trivial with fibres
  $\overline{K}_\bullet(W_{g,1})$.
\end{definition}
\begin{proposition}\label{prop:resolution}
  The map $|X_\bullet(\R^N)| \to \MM_g(\R^N)$ induced by the
  augmentation is $\lfloor (g-3)/2\rfloor$-connected.
\end{proposition}
\begin{proof}
  The map is a locally trivial fibre bundle (at least after replacing
  $\MM_g(\R^N)$ with a compactly generated space) with fibre
  $|\overline{K}_\bullet(W_{g,1})|$ which is $\lfloor
  (g-5)/2\rfloor$-connected, so the claim follows from the long exact
  sequence in homotopy groups.
\end{proof}

The direct limit of $X_p(\R^N)$ as $N \to \infty$ shall be denoted
$X_p$.  These form a semisimplicial space augmented over $\MM_g$, and
the proposition implies that $|X_\bullet| \to \MM_g$ is also $\lfloor
(g-3)/2\rfloor$-connected.  Next we describe the homotopy type of each
space $X_p$.  For example, $X_0$ is the moduli space of manifolds
diffeomorphic to $W_{g,1}$, equipped with an embedding of $H \approx
W_{1,1}$.  Using cancellation
(Corollary~\ref{cor:GeometricCancellation}), it is not hard to
convince oneself that this is weakly equivalent to $\MM_{g-1}$ (for
example, by thinking about the functors classified by the two spaces).
More generally $X_p$ is weakly equivalent to $\MM_{g-p-1}$, but we
need to be precise about the map inducing the homotopy equivalence.

Let $S_1 \subset [0,1] \times \R^{2n}$ be the manifold obtained from
the cylinder $C = [0,1] \times S^{2n-1}$ by forming the connected sum
with an embedded $S^n \times S^n$ along a small disk in $C$.  (For
example, we could embed $S^n \times S^n$ as the boundary of a tubular
neighbourhood of an embedding $S^n \to [0,1] \times \R^{2n}$, although
the precise choice does not matter.)  This comes with a canonical
embedding $W_{1,1} \to S_1$, and we pick an extension of this to an
embedding of $H = W_{1,1} \cup ([0,1] \times D^{2n-1})$, giving an
element $\phi_0 \in \overline{K}_0(S_1)$.  (For defining $\overline{K}_\bullet(S_1)$ we use the
same coordinate patch $c_0: \R^{2n-1} \to \{0\} \times S^{2n-1}
\subset \partial S_1$ as in Definition~\ref{defn:X-p}.)  Similarly,
the $p$-fold concatenation $S_p \subset [0,p] \times \R^{2n}$ of $S_1$
with itself has $p$ canonical embeddings of $W_{1,1}$ which we may
extend to $p$ disjoint embeddings of $H$, giving an element $(\phi_0,
\dots, \phi_{p-1}) \in \overline{K}_{p-1}(S_p)$.  We shall later be slightly more
precise about these choices, but for any such choice we get a map
\begin{align}\label{eq:5}
  \begin{aligned}
    \MM_{g-p} & \lra X_{p-1}\\
    W & \longmapsto (S_p \cup (pe_1+ W), (\phi_0, \dots, \phi_{p-1})),
  \end{aligned}
\end{align}
and the following holds.
\begin{proposition}\label{prop:htpy-type-X-p}
  For $g-p \geq 4$, the map~\eqref{eq:5} is a weak equivalence.
\end{proposition}
\begin{proof}
  There is a restriction map from $X_{p-1}(\R^N)$ to the space of
  embeddings (with fixed behaviour on part of the boundary) of $\{0,
  \dots, p-1\} \times H$ into $[0,\infty) \times \R^N$.  The
  restriction map is a Serre fibration and in the limit $N=\infty$ the
  target is contractible, so $X_{p-1}$ is weakly equivalent to any
  fibre of that fibration, and hence to the subspace of $\MM_g$ which
  consists of manifolds containing the image of the embeddings
  $\phi_0$, \dots, $\phi_{p-1}$.  If we let $A \subset S_p$ denote the
  union of these images and a collar neighbourhood of $\{0\} \times
  S^{2n-1}$, then the inclusion $A \to S_p$ is an isotopy equivalence.
  It follows that $X_{p-1}$ is weakly equivalent to the subspace of
  $\MM_g$ consisting of manifolds containing $S_p$, but by the
  cancellation theorem (Corollary~\ref{cor:GeometricCancellation}),
  this space is in turn equivalent to $\MM_{g-p}$.
\end{proof}
Recall that ``the'' stabilisation map $\MM_{g-1} \to \MM_g$ was
defined by gluing a submanifold of $[0,1] \times \R^N$ diffeomorphic
to $S^n \times S^n$ with two discs cut out.  For large $N$, the space
of such submanifolds is path connected (although not contractible!) so
we get a well defined homotopy class.  To be precise, we shall use the
same submanifold $S_1 \subset [0,1] \times \R^{2n} \subset [0,1]
\times \R^N$ as above.  Together with
Proposition~\ref{prop:htpy-type-X-p}, our next result says that the
last face map of $X_\bullet$ is a model for the stabilisation.
\begin{proposition}\label{prop:d-p-is-stabilisation}
  The following diagram is commutative for $p \geq 0$
  \begin{equation*}
    \xymatrix{
      {\MM_{g-p-1}} \ar[r] \ar[d] & {\MM_{g-p}}\ar[d]\\
      X_{p} \ar[r]_{d_p} & X_{p-1},
    }
  \end{equation*}
  where the vertical maps are given by~\eqref{eq:5} and the top
  horizontal map is the stabilisation map.
\end{proposition}
\begin{proof}
  Starting with $W \in \MM_{g-p-1}$ we map it right to $S \cup (e_1 +
  W) \in \MM_{g-p}$ and then down to the element of $X_p$ given by the
  manifold
  \begin{equation*}
    S_p \cup (pe_1 + (S \cup (e_1 + W))) = S_{p+1} \cup ((p+1)e_1 + W)
    \subset [0,p+1] \times \R^N
  \end{equation*}
  equipped with the embeddings $(\phi_0, \dots, \phi_{p-1})$.  If
  instead we map it down to $X_{p}$, we get the element with the same
  underlying manifold but equipped with the embeddings $(\phi_0,
  \dots, \phi_{p})$, and the face map $d_p:X_p \to X_{p-1}$ then
  forgets $\phi_p$.
\end{proof}
Finally, we want to show that all face maps $d_i: X_p \to X_{p-1}$ are
homotopic, $i = 0, \dots, p$.  For this, we need to be slightly more
precise about the choices of $(\phi_0, \dots, \phi_{p-1}) \in
\overline{K}_{p-1}(S_p)$ used in the map~\eqref{eq:5}.  Firstly, the inclusion
$S_p \to S_{p+1}$ induces a map $\overline{K}_\bullet(S_p) \to
\overline{K}_\bullet(S_{p+1})$ which we may assume sends the $\phi_i \in
\overline{K}_0(S_p)$ to the elements of $\overline{K}_0(S_{p+1})$ with the same names.
Secondly, we may assume that the coordinate patch $c_0: \R^{2n-1} \to
S^{2n-1}$ extends to an embedding $[0,1] \times \R^{2n-1} \to S$ whose
image is disjoint from the canonical embedding $W_{1,1} \to S$.  By
the argument in the proof of
Corollary~\ref{cor:GeometricTransitivity}, we may then pick a
diffeomorphism $\psi: S_2 \to S_2$ supported in the interior of the
complement of the embedded $[0,2] \times \R^{2n-1}$, which
interchanges the two canonical embeddings $W_{1,1} \to S_2$.  We now
pick $\phi_1 \in \overline{K}_0(S_2)$ in the same path component as $\psi \circ
\phi_0$.  (The embedding $\phi_1$ will automatically be equal to $\psi
\circ \phi_0$ when restricted to $W_{1,1} \subset H$ and we choose
the extension to the ``tether'' $[0,1] \times D^{2n-1} \subset H$ by
isotoping what $\psi \circ \phi_0$ does.)  More generally for $p \geq
2$ and $1 \leq i < p$ we let $\psi_{(i-1,i)}$ be the diffeomorphism of
$S_p \subset [0,p] \times \R^{2n}$ which acts as $\psi$ inside
$[i-1,i+1] \times \R^{2n}$ and is the identity outside.  We may then
inductively pick $\phi_p \in \overline{K}_0(S_{p+1})$ in the same path component
as $\psi_{(p-1,p)} \circ \phi_{p-1}$, and such that $\phi_p$ is
disjoint from the images of $\phi_i$ and the support of
$\psi_{(i-1,i)}$ for $i < p$.

\begin{proposition}\label{prop:face-maps-are-homotopic}
  For $0 \leq p \leq g-5$, 
  all face maps $d_i: X_p \to X_{p-1}$ are
  weakly homotopic to one another.
\end{proposition}
\begin{proof}
  Let us focus on the case $p=1$, the general case being similar.  If
  for $i = 0,1$ we write $f_i$ for the composition of $d_i: X_1 \to
  X_0$ with the weak equivalence $\MM_{g-2} \to X_1$ from
  Proposition~\ref{prop:htpy-type-X-p}, we shall construct a homotopy
  $f_0 \simeq f_1: \MM_{g-2} \to X_0$.  These maps are given by the
  formula
  \begin{equation*}
    f_i(W) = (S_2 \cup (2e_1 + W), \phi_i).
  \end{equation*}
  Now the composition of the inclusion $i: S_2 \to [0,2] \times \R^N$
  with the diffeomorphism $\psi: S_2 \to S_2$ is an embedding which
  agrees with $i$ near $\partial S_2$.  For large $N$ the space of
  such embeddings is path connected, so we may find an isotopy of
  embeddings $h_t: S_2 \to [0,2] \times \R^N$ from $i$ to $i \circ
  \psi$, which restricts to the constant isotopy of embeddings of a
  neighbourhood of $\partial S_2$.  Then
  \begin{equation*}
    W \longmapsto (h_t(S_2) \cup (2e_1 + W), h_t \circ \phi_1).
  \end{equation*}
  gives a homotopy of maps $\MM_{g-2} \to X_0$ which starts at $f_1$
  and ends at the map $h_1: W \mapsto (S_2 \cup (2e_1 + W), \psi \circ
  \phi_1)$, but since $\psi \circ \phi_1 \in \overline{K}_0(W_{2,2})$ is in the
  same path component as $\phi_0$, the map $h_1$ is clearly homotopic
  to $f_0$.
\end{proof}

\subsection{The spectral sequence and homological stability}
\label{sec:spectr-sequ-homol}

We now prove Theorem~\ref{thm:main} by induction, using the augmented
simplicial space $X_\bullet$.  The induction hypothesis implies that
for $p > 0$ the face maps $d_i : X_p \to X_{p-1}$ induce homology
isomorphisms in a range, and the induction step will consist of
proving that the augmentation $X_0 \to X_{-1}$ induces a homology
isomorphism in a range which is one larger than the range for $X_2 \to
X_1$.  To prove this, we consider the spectral sequence induced by the
augmented simplicial space $X_\bullet$, with $E^1$ term $E^1_{p,q} =
H_q(X_p)$ for $p \geq -1$ and $q \geq 0$.  The differential is given
by $d^1 = \sum (-1)^i (d_i)_*$, and the group $E^\infty_{p,q}$ is a
subquotient of the relative homology $H_{p+q+1}(X_{-1},
|X_\bullet|)$.
\begin{lemma}\label{lem:form-of-SS}
  We have isomorphisms $E^1_{p,q} \cong H_q(\MM_{g-p-1})$ for $-1 \leq p
  \leq g-5$, with respect to which the differential
  \begin{equation*}
    H_q(\MM_{g-p-1}) \cong E^1_{p,q} \overset{d^1}{\lra} E^1_{p-1,q} \cong
    H_q(\MM_{g-p})
  \end{equation*}
  agrees with the stabilisation map for $p$ even, and is zero
  otherwise.  Furthermore, $E^\infty_{p,q} = 0$ for $p + q \leq
  \lfloor (g-5)/2\rfloor$.
\end{lemma}
\begin{proof}
  Proposition~\ref{prop:htpy-type-X-p} identifies $E^1_{p,q} =
  H_q(X_p) \cong H_q(\MM_{g-p-1})$ and
  Proposition~\ref{prop:face-maps-are-homotopic} shows that all maps
  $(d_i)_*: H_q(X_p) \to H_q(X_{p-1})$ are equal, $i = 0, \dots, p$.
  Therefore all terms in the differential $d^1 = \sum (-1)^i (d_i)_*$
  cancel for $p$ odd, and for $p$ even the term $(d_p)_*$ survives and
  by Proposition~\ref{prop:d-p-is-stabilisation} is identified with
  the stabilisation map.

  The group $E^\infty_{p,q}$ is a subquotient of the relative homology
  $H_{p+q+1}(X_{-1}, |X_\bullet|)$, but this vanishes for $p+q+1 \leq
  \lfloor (g-3)/2 \rfloor$ since the map $|X_\bullet| \to X_{-1}$ is
  $\lfloor (g-3)/2\rfloor$-connected by
  Proposition~\ref{prop:resolution}.
\end{proof}
\begin{proof}[Proof of Theorem~\ref{thm:main}]
  Let us write $a = \lfloor (g-5)/2\rfloor$.  We will use the spectral
  sequence above to prove that $H_q(\MM_{g-1}) \to H_q(\MM_g)$ is an
  isomorphism for $q \leq a$, assuming we know inductively that for $j
  > 0$ the stabilisation maps $H_q(\MM_{g-2j-1}) \to H_q(\MM_{g-2j})$
  are isomorphisms for $q \leq a-j$.  By Lemma~\ref{lem:form-of-SS},
  this implies that the differential $d^1: E^1_{2j,q} \to
  E^1_{2j-1,q}$ is an isomorphism for $0 < j \leq a-q$, and hence that
  $E^2_{p,q} = 0$ for $0 < p \leq 2(a-q)$.  In particular, the
  $E^2_{p,q}$ term vanishes in the region given by $p \geq 1$, $q \leq
  a -1$ and $p + q \leq a+1$, and thus for $r \geq 2$ and $q \leq a$
  it follows that differentials into $E^r_{-1,q}$ and $E^r_{0,q}$
  vanish.  We deduce that for $q \leq a$ we have
  \begin{align*}
    E^\infty_{0,q} = E^2_{0,q} &= \Ker(H_q(\MM_{g-1}) \to
    H_q(\MM_g))\\
    E^\infty_{-1,q} = E^2_{-1,q} &= \Coker(H_{q}(\MM_{g-1}) \to
    H_{q}(\MM_g)),
  \end{align*}
  and since the group $E^\infty_{p,q}$ vanishes for $p+q\leq a$ we see
  that the stabilisation map $H_q(\MM_{g-1}) \to H_q(\MM_g)$ has
  vanishing kernel and cokernel for $q \leq a$, establishing the
  induction step.  The statement is vacuous for $g = 1$ and $g=2$, which starts the induction.
\end{proof}
% \begin{remark}\ 
%   \begin{enumerate}[(i)]
%   \item A similar argument shows that $H_q(\MM_{g-1}) \to H_q(\MM_g)$
%     is surjective for $q = \lfloor (g-3)/2\rfloor$, but this requires
%     the stronger version of Corollary~\ref{cor:GeometricCancellation}
%     from~\cite{Kreck}.  Thus the map in Theorem~\ref{thm:main} is
%     surjective for $g \geq 2k + 2$.
%   \item In the cases $n=3$ and $n=7$, the quadratic module
%     $(\pi_n(W_{g,1}), \lambda, \alpha)$ is just a non-degenerate
%     skew-symmetric form, as $\alpha$ takes values in the trivial
%     group.  In this case, $|K^a(\pi_n(W_{g,1}), \lambda, \alpha)|$ is
%     $\lfloor (g-3)/2\rfloor$-connected by \cite[Theorem
%     4.1]{MR2795243}.  Using this improvement of
%     Theorem~\ref{thm:Charney}, the same argument shows that the map in
%     Theorem~\ref{thm:main} is an isomorphism for $g \geq 2k+2$ (and
%     surjective for $g \geq 2k$) when $n$ is 3 or 7.
%   \end{enumerate}
% \end{remark}
\begin{remark} A similar argument shows that $H_q(\MM_{g-1}) \to
  H_q(\MM_g)$ is surjective for $q = \lfloor (g-3)/2\rfloor$, but this
  requires the stronger version of
  Corollary~\ref{cor:GeometricCancellation} from~\cite{Kreck}.  Thus
  the map in Theorem~\ref{thm:main} is surjective for $g \geq 2k + 2$.
\end{remark}
\begin{remark} In the cases $n=3$ and $n=7$, the quadratic module
  $(\pi_n(W_{g,1}), \lambda, \alpha)$ is just $\bZ^{2g}$ with its
  standard symplectic form, as the quadratic form $\alpha$ takes
  values in the trivial group.  In this case, $|K^a(\pi_n(W_{g,1}),
  \lambda, \alpha)|$ is $\lfloor (g-3)/2\rfloor$-connected by
  \cite[Theorem 4.1]{MR2795243}.  Using this improvement of
  Theorem~\ref{thm:Charney}, the same argument shows that the map in
  Theorem~\ref{thm:main} is an isomorphism for $g \geq 2k+2$ (and
  surjective for $g \geq 2k$) when $n$ is 3 or 7.
\end{remark}

%%% Local Variables: 
%%% mode: latex
%%% TeX-master: "stability"
%%% End: 